\documentclass[12pt]{amsart}
\usepackage{a4wide}
\usepackage{amssymb}
\usepackage{amsthm}
\usepackage{amsmath}
\usepackage{amscd}
\usepackage{verbatim}
\usepackage[all]{xy}

\numberwithin{equation}{section}

\theoremstyle{plain}
\newtheorem{theorem}{Theorem}[section]
\newtheorem{corollary}[theorem]{Corollary}
\newtheorem{lemma}[theorem]{Lemma}
\newtheorem{proposition}[theorem]{Proposition}

\theoremstyle{definition}
\newtheorem{definition}[theorem]{Definition}
\newtheorem{remark}[theorem]{Remark}

\theoremstyle{remark}

\newcommand{\OO}{\mathcal O}
\newcommand{\A}{\mathbb{A}}
\newcommand{\R}{\mathbb{R}}

\newcommand{\Q}{\mathbb{Q}}
\newcommand{\Z}{\mathbb{Z}}

\newcommand{\C}{\mathbb{C}}

\renewcommand{\H}{\mathbb{H}}

\newcommand{\zxz}[4]{\begin{pmatrix} #1 & #2 \\ #3 & #4 \end{pmatrix}}
\newcommand{\abcd}{\zxz{a}{b}{c}{d}}

\newcommand{\kzxz}[4]{\left(\begin{smallmatrix} #1 & #2 \\ #3 & #4\end{smallmatrix}\right) }

\newcommand{\norm}{\operatorname{N}}

\newcommand{\vol}{\operatorname{vol}}
\newcommand{\tr}{\operatorname{tr}}

\newcommand{\Spin}{\operatorname{Spin}}

\newcommand{\Hom}{\operatorname{Hom}}
\newcommand{\Aut}{\operatorname{Aut}}

\newcommand{\End}{\operatorname{End}}

\newcommand{\GL}{\operatorname{GL}}
\newcommand{\SO}{\operatorname{SO}}

\newcommand{\cha}{\operatorname{char}}

\newcommand{\diag}{\operatorname{diag}}
\newcommand{\ord}{\operatorname{ord}}
\newcommand{\kay}{k}
\newcommand{\Gspin}{\operatorname{GSpin}}

\newcommand{\ff}{\hbox{if }}
\newcommand{\SL}{\operatorname{SL}}
\newcommand{\gen}{\operatorname{gen}}

\begin{document}

\title[Quaternions and Kudla's  matching principle ]{Quaternions and Kudla's  matching principle }

\author[]{}
\dedicatory{}

\address{}
\email{}

\thanks{}


\date{\today}
\author[Tuoping Du and Tonghai Yang]{Tuoping Du and Tonghai Yang}
\address{Department of Mathematics, Nanjing University, Nanjing, 210093,  P.R. China}
\email{dtpnju@gmail.com}
\address{Department of Mathematics, University of Wisconsin Madison, Van Vleck Hall, Madison, WI 53706, USA}
\email{thyang@math.wisc.edu} \subjclass[2000]{11G15, 11F41, 14K22}
\thanks{T.Y. Yang is partially supported by a NSF grant DMS1200385 and a Chinese government grant}
\maketitle

\begin{abstract}
In this paper,  we   prove some interesting identities, among average representation numbers (associated to  definite quaternion algebras) and  `degree' of Hecke correspondences on Shimura curves (associated to indefinite quaternion algebras).

\end{abstract}

\maketitle


\section{Introduction} \label{sect:introduction}

In this paper, we   prove some interesting identities relating two
different quaternion algebras using Kudla's matching principle
\cite[Section 4]{KuIntegral}.

Let $D$ be a square square free integer, and let $B=B(D)$ be the unique
 quaternion algebra of discriminant $D$ over $\Q$, i.e., $B$ is
 ramified at a finite prime $p$ if and only if $p|D$.  The reduced
 norm, denoted by $\det$ in this paper,   gives a canonical quadratic form $Q$ on $B$ and makes it a
 quadratic space $V=(B, \det)$.  For a positive integer $N$ prime to
 $D$, let $\OO_D(N)$ be an Eicher order of $B$ of conductor $N$,
 which is an even integral lattice of $V$, and  denote by
 $L$. The quaternion $B$ is definite if and only if $D$ has odd
 number of prime factors.
When $B$ is definite,
 it is a very interesting and hard  question to compute the
 representation number (for a positive integer $m$)
 $$
r_L(m)=|\{ x \in \OO_D(N):\, \det x =
m\}|.
 $$
 On the other hand, the
average over the genus $\gen(L)$, which we denote by
$$
r_{D, N}(m) =r_{\gen(L)}(m) =\bigg(\sum_{L_1 \in \gen(L)}
\frac{1}{|\Aut(L_1)|}\bigg)^{-1} \sum_{L_1 \in \gen(L)}
\frac{r_{L_1}(m)}{|\Aut(L_1)|}
$$
is product of so-called local densities, thanks to Siegel's seminal
work in 1930's \cite{siegel formula}, which are computable (see example
\cite{YaDensity}).  We remark that $\gen(L)$ consists of (equivalence
classes) of  right ideals of all maximal orders when $N=1$.  Notice
that $r_{D, N}$ depends only on $D$ and $N$ and is independent of
the choice of Eicher orders of conductor $N$. Using Kudla's matching principle (\cite{KuIntegral}, see also Section \ref{sect:Kudla}), we will prove the following theorem in Section  \ref{sect:definite}.

\begin{theorem} \label{theo1.1} Let $D$ be a square-free positive
integer with even number of prime factors,  let $p \ne q$ be two
different primes not dividing $D$, and let $N$ be a positive integer
prime to $Dpq$. Then
$$
-\frac{2}{q-1} r_{Dp, N}(m) + \frac{q+1}{q-1} r_{Dp, Nq}(m)
 =-\frac{2}{p-1} r_{Dq, N}(m) + \frac{p+1}{p-1} r_{Dq, Np}(m)
$$
for every positive integer $m$.
\end{theorem}

We remark that $r_{p, N}(m)$ has geometric interpretations. For
example,  Gross and Keating (\cite{GK}, \cite[Page 27]{Wed}) proves
$$
r_{p, 1}(m) =2 \left( \sum_E \frac{1}{|\Aut(E)|}\right)^{-2}
\sum_{(E, E')} \frac{r_{\Hom(E, E')}(m)}{|\Aut(E)| |\Aut(E')|}.
$$
Here the sums are over supersingular elliptic curves over
$\bar{\mathbb F}_p$, and $\Hom(E, E')$ is the quadratic lattice  of
isogenies  (from $E$ to $E'$) with degree as the quadratic form.
Replacing $E$ by a pair $( E,  C)$ where $C$ is a cyclic subgroup of $E$ of order $N$, and $\Hom(E, E')$ by $\Hom( (E, C) , (E', C') )$,
one gets $r_{p, N}(m)$ (see Section \ref{sect:definite} for detail). So Theorem \ref{theo1.1}  gives some relations between supersingular elliptic curves over different primes when we take $D=1$.   We also remark that $|\Aut(E)|$ has a simple formula (see  \cite{Gr2})

When  $B(D)$ is indefinite, the representation number does not make sense anymore as a number can be represented  infinitely many times. In this case, the geometry of Shimura curves comes in.  We fix an embedding of $ i: B(D)\hookrightarrow M_2(\R)$ such that $B(D)^\times $ is invariant under the automorphism  $x \mapsto x^*= {}^tx^{-1}$ of $\GL_2(\R)$. Let   $\Gamma_0^D(N) =\OO_D(N)^1$ be the  group of (reduced) norm $1$ elements in $\OO_D(N)$ and   let  $X_0^D(N) = \Gamma_0^D(N) \backslash \H$ be the associated Shimura curve.  For a positive integer $m$, let $T_{D, N}(m)$ be the Hecke correspondence on $X_0^D(N)$ defined by \begin{equation}
T_{D, N}(m)=\{ ([z_1], [z_2]) \in X_0^D(N) \times  X_0^D(N):\,   z_1 = i(x)z_2  \hbox{ for some } x \in \OO_D(N),   \,   \det x =m\} .
\end{equation}
Define $\deg T_{D, N}(m) =\deg (T_{D, N}(m) \rightarrow X_0^D(N))$  under the  projection $([z_1], [z_2]) \mapsto [z_1]$.

 Let $\Omega_0 =-\frac{1}{4\pi} y^{-2} dx \wedge dy$ be the volume form  on $X_0^D(N)$ (associated to $\Omega_{X_0^D(N)}^\vee$), and let
 $$
 \vol(X_0^D(N), \Omega_0) = \int_{X_0^D(N)} \Omega_0
 $$
 be the volume of $X_0^D(N)$ with respect to $\Omega_0$, which is a  positive rational number (see (\ref{eq:volume})).

Finally,  we define (when $D$ has even number of prime factors)
\begin{equation}
r_{D, N}'(m)= \frac{1}{ \vol(X_0^D(N), \Omega_0)}\deg T_{D,  N}(m).
\end{equation}
Similar to  Theorem \ref{theo1.1}, we will prove the following theorems in Section  \ref{sect:Shimura}.

\begin{theorem}\label{theo1.3}  Let $D$ be a square-free positive
integer with odd number of prime factors,  let $p \ne q$ be two
different primes not dividing $D$, and let $N$ be a positive integer
prime to $Dpq$. Then
$$
-\frac{2}{q-1} r_{Dp, N}'(m) + \frac{q+1}{q-1} r_{Dp, Nq}'(m)
 =-\frac{2}{p-1} r_{Dq, N}'(m) + \frac{p+1}{p-1} r_{Dq, Np}'(m)
$$
for every positive integer $m$.
\end{theorem}

\begin{theorem} \label{theo1.4}  Let $D$ be a square-free positive
integer with odd number of prime factors, let $p \nmid D$ be a prime, and   let $N$ be a positive integer
prime to $Dp$. Then
$$
r_{Dp, N}'(m)= -\frac{2}{p-1} r_{D, N}(m) + \frac{p+1}{p-1} r_{D, Np}(m).
$$

\end{theorem}

\begin{theorem} \label{theo1.5}  Let $D >1$ be a square-free positive
integer with even number of prime factors, let $p \nmid D$ be a prime, and   let $N$ be a positive integer
prime to $Dp$. Then
$$
r_{Dp, N}(m)= -\frac{2}{p-1} r_{D, N}'(m) + \frac{p+1}{p-1} r_{D, Np}'(m).
$$
\end{theorem}
 This paper is organized as follows. In Section \ref{sect:Kudla}, we review Weil representation and Kudla's matching principle in general case.  In Section  \ref{sect:matching}, we prove some local matching between division and matrix quaternion algebras over a local field. In Section \ref{sect:definite} we look at two global quaternions different at two primes carefully and prove Theorem \ref{theo1.1}. In Section  \ref{sect:Shimura}. we associate product of two Shimura curves to the quadratic space coming from an indefinite quaternion and compute the theta integral via `degree' of Hecke correspondence, and prove Theorems \ref{theo1.3}, \ref{theo1.4} and \ref{theo1.5}.

This paper was inspired by Kudla's matching principle. We thank him for his influence and help. We thank Xinyi Yuan for his help to make Section 5 cleaner.  The first author thanks the Department of Mathematics,  University of Wisconsin  at Madison for providing him excellent working and learning environment during his visit. The second author thanks MPIM at Bonn and Tsinghua Math Science Center for providing excellent working conditions during summer 2012, where he did part of this work.

\section{Preliminaries and Kudla's matching principle} \label{sect:Kudla}

Let $\psi: \mathbb{A}/\mathbb{Q} \rightarrow \mathbb{C}$ be the canonical unramified additive character, such that $\psi_{\infty}(x)=e^{2\pi ix}$.
Let $(V, Q)$ be a nondegenerate quadratic space over $\mathbb{Q}$ of even dimension m with the quadratic form $Q$, and let
$$
\chi_V(x) = (x, (-1)^{\frac{m(m-1)}2} \det V)_\A
$$
be the associated quadratic character.
Let  $\omega=\omega_{\psi, V}$ be the associated Weil represenation of $O(V)(\mathbb{A})\times SL_{2}(\mathbb{A})$ on $S(V(\mathbb{A}))$, where $S(V(\mathbb{A}))$ is the Schwartz-Bruhat function space. The orthogonal group $ O(V)(\mathbb{A})$ acts on $ S(V(\mathbb{A}))$ linearly,
\begin{center}
$\omega (h)\varphi (x)=\varphi(h^{-1}x)$.
\end{center}
The $\SL_2(\A)$-action is determined by (see for example \cite{KuSplit})
\begin{eqnarray}\label{weilrep}
&\omega(n(b))\varphi(x)=\psi(bQ(x))\varphi(x), \nonumber\\
&\omega(m(a))\varphi(x)=\chi_V(x) \mid a\mid^{\frac{m}2} \varphi(ax), \\
&\omega(w)\varphi= \gamma(V)  \widehat{\varphi} = \gamma(V)\int_{V(\A)}\varphi(y)\psi((x, y))dy ,\nonumber
\end{eqnarray}
where for $a \in \A^\times$, $b \in \A$
\begin{center}
$n(b)=
\left(
  \begin{array}{cc}
    1 & b \\
     & 1 \\
  \end{array}
\right),
m(a)=
\left(
  \begin{array}{cc}
    a &  \\
     & a^{-1} \\
  \end{array}
\right)
,
w=\left(
  \begin{array}{cc}
     & 1 \\
     -1&  \\
  \end{array}
\right),$
\end{center}
 $dy$ is the Haar measure on $V(\A)$ self-dual with respect to $\psi ((x, y))$, and $\gamma(V)$ is a $8$-th root of unity (Weil index).  Similarly, one has Weil representation at each  prime $p$, which we still denote $\omega$ if there is no confusion. Let  $P=NM$  be the standard Borel subgroup of $\SL_2$, where $N$ and $M$ are subgroups of $n(b)$ and $m(a)$ respectively.  It is well-known that the theta kernel (\cite{Weil})
\begin{equation}
\theta(g, h, \varphi)=\sum_{x\in V(\mathbb{Q})} \omega(g)\varphi(h^{-1}x), \quad \varphi \in S(V(\A))
\end{equation}
is left $O(V)(\Q)\times \SL_2(\Q)$-invariant and is thus an automorphic form on $[O(V)]\times [\SL_2]$. Here denote $[G]=G(\Q) \backslash G(\A)$ for an algebraic group $G$ over $\Q$.  So
the theta  integral
\begin{equation}
I(g, \varphi)=\frac{1}{\vol([O(V)])} \int_{[O(V)]} \theta(g, h, \varphi)dh
\end{equation}
is an automophic form on $[\SL_2]$ if the integral is absolutely convergent, which is the case precisely when  V is antisotropic or  $dim(V)-r>2$, where $ r$ is the Witt index of V. There is another way to construct automorphic forms from $\phi \in S(V(\A))$ via Eisenstein series.

 For $s \in \mathbb{C}$, let $I(s, \chi_V)$ be the principal series representation of $\SL_2(\A)$ consisting of smooth functions $\Phi(s)$ on $\SL_2(\A)$ such that
\begin{equation} \label{eq:1.1}
\Phi(nm(a)g, s) =
\chi_{V}(a)|a|^{s+1}\Phi(g, s).
\end{equation}
There is a $\SL_2(\A)$-intertwining map ($s_0= \frac{m}2 -1$)
\begin{equation}
\lambda=\lambda_{V} : S(V(\mathbb{A})) \rightarrow I(s_{0}, \chi_{V}), \quad \lambda(\varphi)(g)= \omega(g)(0).
\end{equation}
Let $K_{\infty}K$ be the subgroup
$SO_2(\R)\times \SL_{2}(\hat{\mathbb{Z}})$ in $\SL_2(\A)$.  A section $\Phi(s)\in I(s, \chi)$ is called
standard if its restriction to $K_{\infty}K$ is
independent of s. By Iwasawa decomposition  $G(\mathbb{A})=N(\mathbb{A})M(\mathbb{A})K_{\infty}K$, the function
$\lambda(\varphi)\in I(s_{0}, \chi)$ has a unique extension to a
standard section $\Phi(s)\in I(s, \chi)$, where
$\Phi(s_{0})=\lambda(\varphi)$. The Eisenstein series is given by
\begin{equation}
E(g, s, \varphi)= \sum_{\gamma\in P \setminus \SL_2({Q})}
\Phi(\gamma g, s).
\end{equation}
When $V$ is antisotropic or that $dim(V)-r>2$, Kudla and Rallis  (\cite{KR1} \cite{KR2})  proved that the Eisenstein series is holomorphic at $s=s_0$ (extending Weil's classical work \cite{Weil}) and produces an automorphic form $[\SL_2]$.  The two ways (theta integral and Eisenstein series) give the same automorphic form---the well-known Siegel-Weil formula as extended by Kudla and Rallis (\cite{KR1}, \cite{KR2}).

\begin{theorem} (Siegel-Weil formula) \label{theo:Siegel-Weil}
Assume that V is antisotropic or that $dim(V)-r>2$, where r is the Witt index of V, so that the theta  integra is absolutely convergent. Then  $E(g, s; \Phi)$ is holomorphic at the point $s_{0}=m/2-1$, where $m=dim(V)$, and
\begin{center}
$E(g, s_{0}, \Phi)=\kappa I(g, \varphi)$,
\end{center}
 where $\kappa=2$ when $m\leq2$ and $\kappa=1$ otherwise.\\
\end{theorem}

Let $V^{(1)}, V^{(2)}$ be two quadratic spaces with the same dimension and the same character $\chi$, then there is a following graph
\begin{equation}
\setlength{\unitlength}{1mm}
\begin{picture}(60, 20)
\linethickness{1pt}
\put(0,18){$S(V^{(1)}(\mathbb{A}))$}
\put(0,0){$S(V^{(2)}(\mathbb{A}))$}
\put(18,18){ \vector(3,-1){25}}
\put(18,0){ \vector(3,1){25}}
\thicklines
\put(45,8){$I(s_{0},\chi)$}
\put(25,16){$\lambda_{V^{(1)}}$}
\put(25,6){$\lambda_{V^{(2)}}$}
\end{picture}.
\end{equation}

Following Kudla \cite{KuIntegral},  we make the following definition.

\begin{definition} For an prime $p \le \infty$, $\varphi_p^{(i)} \in S(V_p^{(i)})$ are said to be matching if
$$
\lambda_{V_p^{(1)}} (\varphi_p^{(1)}) = \lambda_{V_p^{(2)}}(\varphi_p^{(2)}).
$$
$\varphi^{(i)}=\prod_p \varphi_p^{(i)} \in S(V^{(i)}(\A))$ are said to be matching if they match at each prime $p$.
\end{definition}

By the Siegel-Weil formula, we have the following Kudla matching principle (\cite{KuIntegral}):  Under the assumption of Theorem  \ref{theo:Siegel-Weil} for both $V^{(1)}$ and $V^{(2)}$, one has,   for a
  matching  pair $(\varphi^{(1)}, \varphi^{(2)})$,
\begin{equation} \label{eq:matching}
I(g, \varphi^{(1)})=I(g, \varphi^{(2)}).
\end{equation}

An lattice  $L$ of $V$ is a free $\Z$-submodule of $V$ of rank $m$. It is even integral if $Q(x) \in \Z$ for every $x \in L$. We define  the dual of $L$ as
$$
L^\sharp=\{ x \in  V:\,  (x, L) \subset \Z \}.
$$
Locally,
$$
L_p^\sharp =\{ x \in  V_p :\,  (x, L_p) \subset \Z_p \}.
$$
$L_p$ is called self-dual if $L_p^\sharp =L_p$.

\section{Matchings on quaternions}  \label{sect:matching}

Over a local field $\Q_p$, there are two quaternions: the  matrix algebra  $B^{sp}=M_2(\Q)$ (split quaternion) and the unique division quaternion $B^{ra}$ (ramified  quaternion). Let $V=V^{sp}$ or $V^{ra}$ be the associated quadratic space with reduced norm  as the quadratic form $det(x) =x x^\iota$, where $x^\iota$ is the main involution on quaternion algebra $B$. Both spaces have trivial quadratic character $\chi_V$. So we have $\SL_2(\Q_p)$-intertwining operators
$$
\lambda: S(V) \rightarrow I(1), \varphi \mapsto \lambda(\varphi)=\omega_V(g) \varphi(0).
$$
Here $I(s) =I(s, \hbox{trivial})$. We will use superscript $ra$  and $sp$ to indicate the  association with  division or matrix  quaternion algebra. It is known (\cite{KuIntegral}) that $\lambda^{sp}$ is surjective while the image of $\lambda^{ra}$ is of codimension $1$. So every function $\varphi^{ra}$ has some matching in $S(V^{sp})$. The purpose of this section is to give some explicit matchings and obtain some interesting global identities. In next section, we will give arithmetic and geometric interpretations of these identities in special cases.

\subsection{ The finite prime case $p <\infty$}  We assume $p <\infty$ in this subsection. Let $L^{ra}=\OO_{B^{ra}}$ be the maximal order of $B^{ra}=B_p^{ra}$, which consists of all elements of $B$ whose reduced norm is in $\Z_p$. We don't use the subscript $p$ for simplicity in this subsection.  Let $L_0^{sp}=M_2(\Z_p)$ and
$$
L_1^{sp}= \{ A =\abcd \in M_2(\Z_p):\,  c \equiv 0 \pmod  p \}.
$$
Then
$$
L^{ra, \sharp} =\pi^{-1} L^{ra}, \quad  L_1^{sp, \sharp} = \{ \abcd \in M_2(\Q_p):\,  a, c, d \in \Z_p,  b \in \frac{1}p \Z_p\}.
$$
Here $\pi \in B^{ra}$ is a `uniformizer', i.e., $\pi^\iota =-\pi$ and $ \pi^2 =p$. We denote
$$
\varphi^{ra} =\cha (L^{ra}), \quad  \varphi^{ra, \sharp}= \cha (L^{ra, \sharp}), \quad
$$
and
\begin{equation} \label{eq:varphi}
\varphi_i^{sp} =  \cha(L_i^{sp}),  \quad  i=0, 1,  \quad  \hbox{ and }  \varphi_2^{sp}= \cha(L_1^{sp, \sharp}).
\end{equation}

\begin{proposition}  \label{prop3.1} Let the notation be as above. Then

(1) \quad $ \varphi^{ra} \in S(V^{ra})$ matches with $\frac{-2}{p-1}\varphi_0^{sp}+\frac{p+1}{p-1}\varphi_1^{sp} \in S(V^{sp})$.

(2) \quad $\varphi^{ra, \sharp} \in S(V^{ra})$ matches with $\frac{2p}{p-1}\varphi_0^{sp}-\frac{p+1}{p-1}\varphi_2^{sp} \in S(V^{sp})$.

\end{proposition}
\begin{proof} (1) \quad  Since
$$
\SL_2(\Z_p) = K_0(p) \cup  N(\Z_p) w K_0(p),   \quad   K_0(p) =L_1^{sp} \cap \SL_2(\Z_p),
$$
$I(1)^{K_0(p)}$ has dimension $2$, and $\Phi \in I(1)^{K_0(p)}$ is determined by $\Phi(1)$ and $\Phi(w)$.  Notice that $K_0(p)$ is generated by $n(b)$ and
$n_-(c) = w^{-1} n(-c) w$, $b \in \Z_p$ and $c \in p \Z_p$. Using this, one can check that $\varphi^{ra}$, $\varphi_i^{sp}$ are all $K_0(p)$ under respective Weil representation. We check $\omega(n_-(-c))\varphi^{ra} =\varphi^{ra}$ and leave others to the reader.  One has
$$
\omega^{ra}(w)\varphi^{ra}(x) =\gamma(V^{ra}) \varphi^{ra, \sharp}(x) \vol(L^{ra}).
$$
So
$$
\omega^{ra}(n(-c)w)\varphi^{ra}(x)=\gamma(V^{ra})\vol(L^{ra}) \psi_p(-c \det (x)) \varphi^{ra, \sharp}(x)=\gamma(V^{ra}) \varphi^{ra, \sharp}(x) \vol(L^{ra}),
$$
i.e.,
$$
\omega^{ra}(n(-c)w)\varphi^{ra}= \omega^{ra}(w)\varphi^{ra}.
$$
So
$$
\omega^{ra}(n_-(c)) \varphi^{ra} = \omega^{ra}( w^{-1}) \omega^{ra}(n(-c)w)\varphi^{ra} =\varphi^{ra}
$$
as claimed.

Now we have $\lambda^{ra}(\varphi^{ra}), \lambda^{sp}(\varphi_i^{sp} \in I(1)^{K_0(p)}$. Direct calculation gives
\begin{align*}
\lambda^{ra}(\varphi^{ra})(1) &= 1,  \quad \lambda^{ra}(\varphi^{ra})(w)= \gamma(V^{ra}) p^{-1}
\\
\lambda^{sp}(\varphi_0^{sp})  &= 1,  \quad  \lambda^{sp}(\varphi_0^{sp})(w)  = \gamma(V^{sp}),
\\
\lambda^{sp}(\varphi_1^{sp})  &= 1,  \quad  \lambda^{sp}(\varphi_0^{sp})(w)  = \gamma(V^{sp})p^{-1}.
\end{align*}
Since  $\gamma(V^{sp})=-\gamma(V^{ra}) (=1)$, one has
$$
\lambda^{ra}(\varphi^{ra})= \frac{-2}{p-1} \varphi_0^{sp}  + \frac{p+1}{p-1} \varphi_1^{sp}.
$$
This proves (1). Claim (2) is similar and is left to the reader. One just needs to replace $K_0(p)$ by
$$
K_0^+(p) = \{ \abcd \in \SL_2(\Z_p):\,  b \equiv 0 \pmod  p \}.
$$

\end{proof}

To find  matching  pairs of  coset functions $\varphi_\mu^{ra} =\cha (\mu +L^{ra})$, we first need to label them. Let $\kay$ be the unique unramified quadratic field extension of $\Q_p$ in $B^{ra}$, and let $\OO_\kay=\Z_p + \Z_p u$ be the ring of integers of $\kay$ with $u \in \OO_\kay^\times$. Then there is a uniformizer $\pi$ of $B$ such that $\pi r = \bar r \pi$ for $r \in \kay$ and  $\pi^2 =p$. One has then
$$
L^{ra} =\OO_{B^{ra}} = \OO_\kay + \OO_\kay \pi = \Z_p + \Z_p u + \Z_p \pi + \Z_p u p.
$$
So one has an isomorphism
$$
(Z/p)^2 \cong L^{ra, \sharp}/L^{ra},  \quad  (i, j) \mapsto \mu^{ra}_{i, j} = \frac{i+ j \mu}{\pi}.
$$
Using the identification, we denote $\varphi_{i,j}^{ra} $ for $\cha(\mu^{ra}_{i, j} + L^{ra})$.  Similarly, we use $\varphi_{i, j}^{sp}$ to denote
$\cha( \mu^{sp}_{i, j}  + L_1^{sp})$, where  $\mu^{sp}_{i, j} =\kzxz {0} {\frac{j}p} {i} {0}$.   Let
$$
K(p) =\{ \abcd \in \SL_2(\Z_p):\,  a-1 \equiv d-1 \equiv b \equiv c \equiv 0 \pmod p\}.
$$
Since
\begin{equation}\label{decomposition}
N(\Z_p) M(\Z_p) \backslash \SL_2(\Z_p)/K(p)  = \{1,  w n(j),  0 \le j \le p-1\},
\end{equation}
one has $\dim I(1)^{K(p)} = p+1$. 

\begin{lemma}\label{lem:basis} Let the notation be as above. Then 

 (1) \quad  One  has $\varphi_{i, j}^{sp}, \varphi_{i, j}^{ra} \in I(1)^{K(p)}$.

(1) \quad  When $ab \equiv cd \pmod p$ and $(a, b), (c, d) \ne (0, 0)$, one has $\lambda^{sp}(\varphi_{a, b}^{sp}) = \lambda^{sp}(\varphi_{c, d}^{sp})$.

(2) \quad  The set $\{ \lambda^{sp}(\varphi_0^{sp}),  \lambda^{sp}(\varphi_{1, j}^{sp}), 0 \le j \le p-1 \} $  gives a basis of $I(1)^{K(p)}$.

\end{lemma}
\begin{proof} Claim (1) follows from definition and the fact that $K(p)$ is generated by $n(pb)$ and $n_-(pb)$, $b \in \Z_p$. The invariant under $n(pb)$ is clear.  The invariant under $n_-(pb)$ can be verified the same way as in the proof of Proposition \ref{prop3.1}.

(2) \quad  By  (\ref{decomposition}), it is only need to check the values at $\{1, w n(i),  0 \le i \le p-1\}$.
\begin{eqnarray}
\lambda^{sp}(\varphi^{sp}_{a, b})(w n(i))&=& \int_{{\mu_{a, b}}+\mathcal{O}_{p}} \psi_{p}(i\det(x))dx \nonumber\\
&=& \frac{1}{p}e(abi/p)\nonumber
\end{eqnarray}
and $$ \lambda^{sp}(\varphi^{sp}_{a, b})(1)=0,$$
where $e(x)=e^{2\pi \sqrt{-1} x}$.
The result follows.\\
(3) \quad
By (\ref{decomposition}), we
see  that $\Phi\in I(1)^{K(p)}$ is
determined by the values at
\begin{center}
$\{1, w n(i),  0 \le i \le p-1\}.$
\end{center}

Suppose
$$a\lambda^{sp}(\varphi_0^{sp})+ \sum_{0\leq j\leq p-1}a_{j}\lambda^{sp}(\varphi_{1, j}^{sp})=0,$$
where $a, a_{j}\in \mathbb{C}$.\\

Taking  the value at 1:
$$a\lambda^{sp}(\varphi_0^{sp})(1)+ \sum_{0\leq j\leq p-1}a_{j}\lambda^{sp}(\varphi_{1, j}^{sp})(1)=0,$$
we get a=0. Now we consider the equations

$$ \sum_{0\leq j\leq p-1}a_{j}\lambda^{sp}(\varphi_{1, j}^{sp})(wn(i))=0.$$
The coefficient matrix is $\frac{1}p \mathbf{A}:=\big[\lambda^{sp}(\varphi_{1, j}^{sp})(wn(i))]_{0\leq i,j\leq p-1}.$  Since
$$
\lambda^{sp}(\varphi_{1, j}^{sp})(wn(i))=\frac{1}{p}e(ij/p)
,$$ one sees  $\mathbf{A}=(e(ij/p))_{0\leq i,j\leq p-1}$ and
 $\det(\mathbf{A})= \det((e(ij/p)))\neq0$.\\

Hence $a=0, a_{j}=0, 0\leq j\leq p-1$,  and $ \lambda^{sp}(\varphi_0^{sp}),  \lambda^{sp}(\varphi_{1, j}^{sp}), 0 \le j \le p-1  $ are  linear independent.
\end{proof}

\begin{proposition} \label{prop3.3}  Assume $(k, l) \ne 0$. Let $\mathbf A=(e(\frac{ij}p))_{0 \le i, j\le p-1}$  be  the matrix in the proof of Lemma \ref{lem:basis},  and let $\mathbf A_j$  be the matrix  obtained by replacing $j$-th column of $\mathbf A$ by column $\{- e(\frac{ -i d_{k, l}}p), 0 \le i \le p-1\}$, where
 $$
 d_{k, l}= k^2 + kl \hbox{Tr}_{\kay/\Q_p} (u) + l^2 \norm_{\kay/\Q_p} (u).
 $$
 Then $\sum_{j=0}^{p-1} c_{k, l}(j) \varphi_{1, j}^{sp} \in S(V^{sp})$ is matching with $\varphi_{k, l}^{ra}$, where
$c_{k, l}(j)=\frac{det\mathbf{A}_{j}}{det{\mathbf{A}}}$.
\end{proposition}
\begin{proof}
Since  $\lambda^{ra}(\varphi_{k, l}^{ra}) \in I(1)^{K(p)}$,  it suffices to check the identity at $\{1, wn(i), 0\le i \le p-1\}$.

Suppose $(k, l)\neq0 $, let
\begin{equation}
\lambda^{ra}(\varphi^{ra}_{k, l})= b_{k, l}\lambda^{sp}(\varphi_{0}^{sp})+\sum_{0\leq j \leq p-1}c_{k, l}(j) \lambda^{sp} (\varphi^{sp}_{1,j}),\nonumber
\end{equation}
where $b_{k, l}, c_{k, l}(j) \in \mathbb{C}.$     Taking the value at $1$, one  gets  $b_{k, l}=0$.

Taking the value at $\{wn(i)$, $0\leq i\leq p-1\}$, one has
\begin{equation}\nonumber
\lambda^{ra}(\varphi^{ra}_{k, l})(wn(i))= \sum_{0\leq j \leq p-1}c_{k, l}(j) \lambda^{sp} (\varphi^{sp}_{1, j})(wn(i)).
\end{equation}
It is easy to check that $$\lambda^{ra}(\varphi^{ra}_{k, l})(wn(i))=-\frac{1}{p}e(-\frac{i}{p}(k^{2}+klTr(u)+l^{2}N(u))).$$ Denoting $$
 d_{k, l}= k^2 + kl \hbox{Tr}_{\kay/\Q_p} (u) + l^2 \norm_{\kay/\Q_p} (u),
 $$ then $\lambda^{ra}(\varphi^{ra}_{k, l})(wn(i))=-\frac{1}{p}e(-\frac{id_{k, l}}{p})$.
From the proof of Lemma \ref{lem:basis}, it is known that $\lambda^{sp} (\varphi^{sp}_{1, j})(wn(i))=\frac{1}{p}e(ij/p)$.
So we get $c_{k, l}(j)=\frac{det\mathbf{A}_{j}}{det{\mathbf{A}}}$.

\end{proof}

\subsection{The case $p=\infty$} In this subsection we  consider the case $\Q_p=\R$ and  recall a matching pair given in \cite{KuIntegral}.  Notice that $B^{ra}$ in this case is the Hamilton division algebra, and $V^{ra}$ has signature $(4, 0)$. Let $\varphi_\infty^{ra} (x) =e^{ - 2 \pi \det (x)} \in S(V^{ra})$, then $\varphi_\infty^{ra}$ is of weight $2$ in the sense
$$
\omega^{ra}(k_\theta) \varphi_\infty^{ra} = e^{2 i \theta}   \varphi_\infty^{ra}, \quad k_\theta = \kzxz {\cos\theta} {\sin\theta} {-\sin\theta} {\cos\theta}.
$$
On the other hand, Kudla constructed a family of  weight $2$ Schwartz function $\varphi_\infty^{sp} \in S(V^{sp})$  as follows \cite[Section 4.8]{KuIntegral}. Recall $V^{sp} = M_2(\R)$. Given an orthogonal decomposition
\begin{equation} \label{eq:spacedecomposition}
V^{sp} = V^+ \oplus V^-,  \quad x = x^+ + x^-,
\end{equation}
with $V^+$ of signature $(2,0)$ and $V^-$ of signature $(0, 2)$. One defines (Kudla used the notation $\tilde\phi(x, z)$)
$$
\varphi_\infty^{sp} (x, V^-) = (4\pi (x^+, x^+) -1) e^{ -\pi (x^+, x^+) + \pi (x^-, x^-)}.
$$
Kudla proved the following proposition \cite[Section 4.8]{KuIntegral}.

\begin{proposition} \label{prop3.4}   For any orthogonal decomposition as in  (\ref{eq:spacedecomposition}), $(\varphi_\infty^{ra}, \varphi_\infty^{sp}(x , V^-))$ is a matching pair, and their (same) image in $I(1)$ is the unique weight $2$ section $\Phi_\infty^2$ given by
$$
\Phi_\infty^2(n(b) m(a) k_\theta)= |a|^2 e^{2 i \theta}.
$$
\end{proposition}

Because of this proposition, we will simply use  $\varphi_\infty^{sp}$ for $\varphi_\infty^{sp} (\, ,  V^-)$.

\subsection{Global matching} For a square-free  positive integer, and let $B(D)$ be the unique quaternion algebra over $\Q$ of discriminant $D$ as in the introduction. It is indefinite (i.e., $B(D)_\R \cong M_2(\R)$) if and only if $D$ has even number of prime factors and it is $M_2(\Q)$, i.e., $D=1$,  precisely when it represents $0$.
In this paper, we assume $D >1$ so the Siegel-Weil formula applies.  The following matching theorem is clear from Kudla's matching principle (\ref{eq:matching}) and Propositions \ref{prop3.1}, \ref{prop3.3}, and  \ref{prop3.4}.

\begin{proposition}  \label{prop3.5} Let $V(D_i)$ be the quadratic spaces associated to  the quaternion algebras $B(D_i)$ (with reduced norm as the quadratic form) with  square-free integers $D_i >1$, $i=1, 2$. Assume $\varphi^{(i)}=\prod_p \varphi_p^{(i)} \in S(V(D_i)(\A))$ satisfy the following conditions:

(1) \quad When $p =\infty$,   $\varphi_\infty^{(i)}$  is $\varphi_\infty^{sp}$ or $\varphi_\infty^{ra}$ depending on whether $B(D_i)_\infty$ is  split or non-split.

(2) \quad When $p \nmid D_1 D_2 \infty$ or $p | \hbox{gcd}(D_1, D_2)$, we identify $V(D_1)_p = V(D_2)_p$ and take any $\varphi_p^{(1)}= \varphi_p^{(2)} \in S(V(D_1)_p)$.

(3) \quad When $ p|\hbox{lcm}(D_1, D_2)$ but $p\nmid \hbox{gcd}(D_1, D_2)$, one of $B(D_i)$ is $B_p^{sp}$ and the one is $B_p^{ra}$, we take $(\varphi_p^{(1)}, \varphi_p^{(2)})$ to be a matching pair in Propositions \ref{prop3.1} and \ref{prop3.3}.

Then $(\varphi^{(1)}, \varphi^{(2)})$ is a matching pair, and
$$
I(g, \varphi^{(1)}) = I(g, \varphi^{(2)}), \quad g \in \SL_2(\A).
$$
\end{proposition}

In next two sections, we will give arithmetic and interpretations of the theta integrals in some special cases.

\section{Definite quaternions, representations numbers, and supersingular ellipti curves }  \label{sect:definite}

We first review a general fact about positive definite quadratic forms for the convenience of the reader.  Let $(V, Q)$ be a positive definite quadratic space
of even dimension $m$. Define
$$
\varphi_\infty(x) = e^{-2 \pi Q(x)}  \in S(V_\infty).
$$
Then  it has the properties
$$
\varphi_\infty(hx) =\varphi_\infty(x), \quad \omega(k_\theta) \varphi_\infty = e^{\frac{m}2 i \theta} \varphi_\infty
$$
for  $h \in O(V)(\R)$ and $k_\theta \in \SO_2(\R) \subset \SL_2(\R)$. For any $\varphi_f \in S(\hat V)$, where $\hat V = V \otimes_\Z \hat\Z$,  the theta kernel
$$
\theta(\tau, h, \varphi_f \varphi_\infty) = v^{-\frac{m}2} \theta(g_\tau, h, \varphi_f \varphi_\infty)
$$
is a holomorphic modular form of weight $\frac{m}2$ for some congruence subgroup, so is
$$
I(\tau, \varphi_f\varphi_\infty) = v^{-\frac{m}2} I(g_\tau, \varphi_f\varphi_\infty).
$$
Here $g_\tau = n(u) m(\sqrt v)$ for $\tau =u + i v \in \mathbb H$.

For an even integral lattice $L$ of $V$, we denote
\begin{equation} \label{eq:new4.1}
\theta(\tau, L) = \theta(\tau,  \cha(\hat L) \varphi_\infty), \quad I(\tau, L)=I(\tau,  \cha(\hat L) \varphi_\infty).
\end{equation}
Recall that two lattices $L_1$ and $L_2$  of $V$ are equivalent if there is $h \in O(V)$ such that $hL_1 =L_2$. Two lattices $L_1$ and $L_2$  are in the same genus if they are equivalent locally everywhere, i.e,  there is $h \in O(V) (\hat\Q)$ such that $h\hat L_1= \hat L_2$.  We also recall that $O(V)(\A)$ acts on the set of lattices as follows: $h L =  (h_f \hat L)\cap V$ where $h_f$ is the finite part of $h=h_f h_\infty$.  Let $\gen(L)$ be the genus of $L$---the set of equivalence classes of lattices in the same genus of $L$. Then the above discussion implies that
$$
O(V)(\Q) \backslash O(V)(\A)/K(L) O(V)(\R) \cong  \gen(L), \quad [h] \mapsto hL,
$$
where $K(L)$ is the stabilizer of $\hat L$ in $O(V)(\hat\Q)$.

\begin{proposition} Let
$$
r_L(n) =|\{ x \in  L:\,  Q(x) =n\}|, \quad  r_{\gen(L)}(n) = \left(\sum_{L' \in \gen(L)} \frac{1}{|O(L')|}\right)^{-1} \sum_{L' \in \gen(L)}  \frac{r_{L'}(n)}{|O(L')|},
$$
where $O(L)$ is the stabilizer of $L$ in $O(V)$. Then ($q =e(\tau)$)
\begin{align*}
\theta(\tau, h,  L) &=\sum_{n=0}^\infty r_{hL}(n) q^n,
\\
I(\tau, L) &= \sum_{m=0}^\infty r_{\gen(L)}(n) q^n.
\end{align*}
\end{proposition}
\begin{proof} (sketch) This is well-known and we sketch the main steps for the convenience of the reader. The formula for $\theta$  follows directly from the definition. For theta integral, notice that $\cha(\hat L)$ is $K(L)$-invariant. Write
$$
O(V)(\A) = \cup_{j=1}^r O(V)(\Q) h_j K(L) O(V)(\R).
$$
Then $\gen(L) =\{ h_j L:\, j=1, \cdots, r\}$, and
\begin{align*}
\vol([O(V)]) I(\tau, L) &= \sum_j \theta(\tau, h_j, L) \int_{ (h_j^{-1} O(V)(\Q) h_j)\cap K(L) \backslash K(L)O(V)(\R)} 1 dh
\\
 &= \vol(K(L) O(V)(\R)) \sum_j  \frac{\theta(\tau, h_j, L)}{|O(h_jL)|}
 \\
 &= \vol(K(L) O(V)(\R))  \sum_{n=0}^\infty (\sum_{L' \in \gen(L)} \frac{r_{L'}(m)}{|O(L')|}) q^m.
\end{align*}
On the other hand, the same calculation gives
$$
\vol([O(V)]) = \vol(K(L) O(V)(\R)) \sum_{L' \in \gen(L)} \frac{1}{|O(L')|}.
$$
This proves the formula for $I(\tau, L)$.
\end{proof}

{\bf Proof of Theorem  \ref{theo1.1}}:  Let $V(D)$ be the quadratic space associated to the   quaternion algebra $B(D)$ of discriminant $D$. Recall that a Eicher order of conductor $N$, denoted by $\OO_D(N)$,  is an order $O$ of $B(D)$ such that
\begin{enumerate}
\item  When $p|D$, $O_p:=O \otimes_\Z \Z_p$ is the maximal order of $B(D)_p=B_p^{ra}$.

\item When $p\nmid D\infty$,   there is an identification $B(D)_p \cong M_2(\Q_p)$ under which $\OO_D(N)_p =R_p(N)$. Here
$$
R_p(N)= \{ \abcd \in M_2(\Z_p): \,  c \equiv 0 \mod N\}.
$$
\end{enumerate}
Now let $V^{(1)}=V(Dp)$ and $V^{(2)}=V(Dq)$ with $D$ satisfying the condition in the theorem. Then $V^{(i)}$ are both positive definite. Define $\varphi^{(1)} =\prod_l \varphi_l^{(1)} \in S(V^{(1)}(\A))$ as follows.
$$
\varphi_l^{(1)} =\begin{cases}
  \varphi_\infty^{ra} &\ff l =\infty,
  \\
 \cha( R_l(N))  &\ff   l \nmid Dq,
 \\
 \varphi_{l}^{ra}     &\ff l | Dp,
 \\
 \frac{-2}{q-1} \varphi_{q,0}^{sp} + \frac{q+1}{q-1} \varphi_{q, 1}^{sp}   &\ff  l=q
 \end{cases}
$$
where $\varphi_{l, i}^{sp}$ and $\varphi_l^{ra}$ are the functions defined in  (\ref{eq:varphi}) with added subscript $l$ (to indicate its independence). Then one has
$$
\varphi_f^{(1)} = \frac{-2}{q-1} \cha(\hat{\OO}_{Dp}(N)) + \frac{q+1}{q-1}\cha(\hat{\OO}_{Dp}(Nq)).
$$
So
$$
I(\tau, \varphi^{(1)}) = \frac{-2}{q-1} I(\tau, \OO_{Dp}(N)) + \frac{q+1}{q-1}I(\tau, \OO_{Dp}(Nq)).
$$
 One defines  $\varphi^{(2)}$ the same way with  the roles of $p$ and $q$ switched. Then $\varphi^{(1)}$ and $\varphi^{(2)}$ form a matching pair by Proposition \ref{prop3.3}. So Proposition \ref{prop3.5} implies
 $$
 I(\tau, \varphi^{(1)}) =I(\tau, \varphi^{(2)})
 $$
and thus
$$
\frac{-2}{q-1} I(\tau, \OO_{Dp}(N)) + \frac{q+1}{q-1}I(\tau, \OO_{Dp}(Nq)) = \frac{-2}{p-1} I(\tau, \OO_{Dq}(N)) + \frac{p+1}{p-1}I(\tau, \OO_{Dq}(Np)).
$$
Taking  $m$-th Fourier coefficients, one proves the theorem.

The case $D=1$ has special geometric meaning as indicated in the introduction.  Let $Y_0(N)$ be the moduli stack of pairs $( E, C)$ where $E$ is an elliptic curve and $C$ is a cyclic sub-scheme of order $N$. It is regular and flat over $\Z$ and smooth over $\Z[\frac{1}N]$.  For a prime  $p\nmid N$, let $SS_p(N)$ be the supersingular locus of $Y_0(N)(\bar{\mathbb F}_p)$---the $\bar{\mathbb F}_p$-point $( E, C)$ such that  $E$ is  supersingular, i.e, $\End(E)$ and $\End(E')$ are maximal orders of $B(p)$. In this case, the endomorphism ring $\End(E, C)$  is an Eichler order $\OO_p(N)$ of conductor $N$. Every Eichler of $B(p)$ comes this way. For two points $(E_1, C_1), (E_2, C_2)  \in SS_p(N)$,
$
\Hom((E_1, C_1), (E_2, C_2))
$,
 which consists of isogenies  $(f: E_1 \rightarrow E_2)$ with $ f(C_1) \subset C_2 $, is an quadratic lattice with respect to $\deg f $, and is in the same genus of $\End(x_1)$ and $\End(x_2)$. One can  actually prove (see example \cite{YaGeneral}) that all $\Hom(x_1, x_2)$ form the genus of $L=\OO_p(N)$ as $x_1$ and $x_2$ runs through the supersingular locus $SS_p(N)$.   So we have

 \begin{proposition} One has
 $$
 r_{p, N}(m) =\left(\sum_{x_1, x_2 \in SS_p(N)}\frac{1}{|\Aut(x_1)| |\Aut(x_2)|}\right)^{-1}  \sum_{x_1, x_2 \in SS_p(N)}\frac{ r_{\Hom(x_1, x_2)} (m)}{|\Aut(x_1)| |\Aut(x_2)|}.
 $$
 \end{proposition}

\section{Indefinite quaternions and Shimura curves} \label{sect:Shimura}

Associated to a square-free integer $D>0$ with even number of prime factors, is an indefinite quaternion algebra $B=B(D)$ of discriminant $D$. In particular, We choose and fix an embedding  $i: B  \hookrightarrow B_\infty \cong M_2(\R)$ such that the inner isomorphism $X \mapsto w X w^{-1}$ preserves $i(B)$.   We denote $\det$ for the reduced norm on $B$, then $V= V(D) =(B, \det)$ is of signature $(2, 2)$ and is anisotropic when $D>1$. According to \cite[Theorem 4.23]{KuIntegral},  the theta integral $I(g, \varphi)$ in  Proposition  \ref{prop3.5}  is a generating function of degrees of  some devisors with respect to the tautological line bundle over the Shimura variety associated to $V$. In our case, the line bundle can be identified with the line bundle of two variable modular forms of weight $1$, and the devisors can be identified with Hecke correspondences on a Shimura curve as we will see now.

 The action of  $B^\times \times B^\times$  on $V$ via
$$
(g_1, g_2)X = g_1 X g_2^{-1}
$$
gives  an identification of  $\Gspin(V)$ with
$$
H=\{ (g_1, g_2) \in  B^\times \times B^\times:\,  \det g_1 = \det g_2 \}.
$$
The associated spin norm is $\mu(g_1, g_2) =\det g_1$.
It has the exact sequence
$$
1 \rightarrow \mathbb G_m \rightarrow H \rightarrow \SO(V) \rightarrow 1.
$$
Let  $\mathbb D$ be the Hermitian domain of oriented negative $2$-planes in $V_\R$, and
$$
\mathcal L = \{ w \in V_\C=M_2(\C):\,  (w, w) =0,  (w, \bar w) <0\}
$$
on both of which $H(\R)$ acts naturally.
The  map
$$
f: \mathcal L/\C^\times \cong  \mathbb D,  \quad w= u+ i v \mapsto  \R (-u) + \R v
$$
gives an $H(\R)$-equivariant isomorphism between $\mathcal L/\C^\times$ with $\mathbb D$. Thus $\mathcal L$ is a  (tautological) line bundle over $\mathbb D$. The Hermitian domain has also a tube representation which we will need.  Indeed, the map
$$
\mathbf w:  (\mathbb H)^2  \times (\mathbb H^-)^2 \rightarrow \mathcal L,  \quad \mathbf w(z_1, z_2) = \kzxz {z_1 z_2} {z_1} {z_2} {1},
$$
gives an isomorphism
$$
 (\mathbb H)^2  \times (\mathbb H^-)^2 \cong  \mathcal L/\C^\times  \cong  \mathbb D.
$$
We will identify $\mathbb D$ with $(\mathbb H)^2  \times (\mathbb H^-)^2$ via this isomorphism.  The natural action of $B^\times \times B^\times $  on $V$ induces the following action on $(\mathbb H)^2  \times (\mathbb H^-)^2$:
\begin{equation}
(g_1, g_2)(z_1, z_2) = (i(g_1)z_1, i(g_2)^* z_2)
\end{equation}
where $g^* = {}^tg^{-1}$ for $g \in \GL_2(\R)$.
 One also has
\begin{equation}
(g_1, g_2)w(z_1, z_2) = w( (g_1, g_2)(z_1, z_2)) (c_1 z_1+d_1) (c_2 z_2 +d_2)
\end{equation}
for
$$
g_1 = \kzxz {a_1} {b_1} {c_1} {d_1} \in H(\R) , \quad  g_2^*= \kzxz {a_2} {b_2} {c_2} {d_2} \in H(\R)
$$

Associated to  a compact open subgroup $K$ of $H(\hat Q)$ is a Shimura variety $X_K$ over $\Q$ such that
$$
X_K(\C) = H(\Q) \backslash \mathbb D \times H(\hat\Q)/K.
$$
Moreover, $\mathcal L$ descends to a line bundle on $X_K$, which we continue to denote by $\mathcal L$. It can be identified with the line bundle of two variable modular forms of weight $(1, 1)$.   In this section, we always assume
$$
K = \{ (k_1, k_2) \in \hat{\OO}_D(N)^\times \times  \hat{\OO}_D(N)^\times:\,   \det k_1= \det k_2 \} \subset H(\hat\Q)
$$
which preserves the lattice $L=\OO_D(N)$.
\begin{lemma} \label{lem5.1} Let the notation be as above. Then one has an isomorphism
$$
 X_0^D(N) \times X_0^D(N) \cong X_K,  ([z_1], [z_2]) \mapsto [z_1, wz_2]
$$
where $X_0^D(N) = \Gamma_0^D(N)\backslash \mathbb H$ is the Shimura curves  defined in the introduction (recall $w=\kzxz {0} {1} {-1} {0}$).
\end{lemma}
\begin{proof}Let
$$
H_1=\{ (g_1, g_2) \in H:\,  \det g_1 =\det g_2 =1\} =\ker \mu =\Spin(V),   \quad K_1 =H_1(\hat\Q) \cap K.
$$
By the strong approximation theorem, one has
$$
H_1(\A) = H_1(\Q) K_1 H_1(\R).
$$
Since $\mu(H(\Q) K H(\R)^+)=\A^\times$, one has then
$$
H(\A) =H(\Q) K H(\R)^+.
$$
So
\begin{align*}
X_K &= H(\Q) \backslash H(\A)/(K K_\infty)
\\
&= H(\Q) \backslash (H(\Q) K H(\R)^+)/(K K_\infty)
\\
 &=( H(\Q)\cap (K H(\R)^+ )) \backslash H(\R)^+/K_\infty.
\end{align*}
Here $K_\infty$ is stabilizer of  $(i, i) \in H^2$ in $H(\R)$  and also  in $H(\R)^+$. Notice that
$$
H(\Q)\cap (K H(\R)^+ ) = H_1(\Q) \cap K_1 =\Gamma_0^D(N) \times \Gamma_0^D(N).
$$
So
$$
X_K = X_0^D(N) \times X_0^D(N)^*,
$$
where $X_0^D(N)^* =\Gamma_0^D(N) \backslash \H$ with a slightly different action $\gamma*z= \gamma^*(z)$. Now the lemma follows from the isomorphism
$$
X_0^D(N)  \cong X_0^D(N)^*,  \quad [z] \mapsto [wz].
$$

\end{proof}

Let
$$
\Omega = -\frac{1}{4 \pi} \left( y_1^{-2}  dx_1\wedge dy_1 +  y_2^{-2}  dx_2\wedge dy_2\right)
$$
be as in \cite[Example 4.13]{KuIntegral}. It corresponds to Chern class $-c_1(\mathcal L)$ in $H^2(X_K)$.

Next, we describe the Kudla cycle on $X_K$ and relate it to Hecke correspondence on $X_0^D(N)$. Let $\Omega_0 = -\frac{1}{4 \pi} y^{-2} dx\wedge dy$ be as in the introduction, and let $\pi_1$ and $\pi_2$ be two natural projections of $X_K =X_0^D(N) \times X_0^D(N)$ onto $X_0^D(N)$. Then
$$
\Omega=  (\pi_1^*(\Omega_0) + \pi_2^*(\Omega_0)).
$$
Moreover,
one has by  \cite[(2.7)]{KRYComp} and
\cite[Lemma 5.3.2]{Miy}
\begin{align} \label{eq:volume}
\vol(X_0^D(N), \Omega_0) &:= \int_{X_0^D(N)} \Omega_0 = [\OO_D^1 : \Gamma_0^D(N)] \zeta_D(-1)
\\
   &=-\frac{ DN}{12}\prod_{p|N} (1+p^{-1}) \prod_{p|D} (1-p^{-1} )  \in \frac{1}{12} \Z_{<0}  ,  \notag
\end{align}
where $\zeta_D(s) =\prod_{p\nmid D} (1-p^{-s})^{-1}$ is the partial zeta function, and $\OO_D$ is a maximal order of $B$ containing $\OO_D(N)$.

For an $x \in V$ with $\det (x) >0$ and $h \in H(\hat Q)$, $x^\perp$
is of signature $(1, 2)$ and defines a sub-Shimura variety $Z(x)$ of $X_{h K h^{-1}}$, its right translate by $h$ gives a divisor $Z(x, h)$ in $X_K$.
For $\varphi_f \in S(\hat V)^K$ and $m \in \Q_{>0}$, one defines the associated Kudla cycle  $Z(m, \varphi_f)$  as
$$
Z(m, \varphi_f)= \sum_{j=1}^r  \varphi_f(h^{-1} x) Z(x, h)
$$
if there is $x \in V$ such that $\det (x) =m$ and
$$
\hbox{Supp}(\varphi_f) \cap \{ x \in V(\hat\Q):\,  \det x =m \} = \coprod_{j=1}^r  K h_j^{-1} x.
$$
Otherwise, we defines $Z(m, \varphi_f) =0$.

\begin{lemma}  \label{lem5.2} Let $T_{D, N}(m)$ be the Hecke operator  on $X_0^D(N)$ as in the introduction. Then (under the identification $X_K \cong X_0^D(N) \times X_0^D(N)$ in Lemma \ref{lem5.1})
$$
Z(m , \cha(\hat L)) = T_{D, N}(m)
$$
where $L =\OO_D(N)$.
\end{lemma}
\begin{proof} Let $L_m =\{ x \in L:\,  \det x =m\}$.
By proof of Lemma \ref{lem5.1}, one has
$H(\hat\Q) =H(\Q)K$.  So in the decomposition ($x \in V$ with $Q(x) =m$)
$$
\hat L_m= \coprod K h_j^{-1} x
$$
we may assume $h_j \in  H(\Q)$. This implies
$$
L_m = \coprod \Gamma_K h_j^{-1} x = \coprod \Gamma_K x_j,  \quad x_j =h_j^{-1} x \in L,
$$
where $\Gamma_K=K \cap H(\Q)$,
and
$$
Z(m, \varphi_f) = \sum_j Z(x, h_j) =\sum_j  Z(h_j^{-1} x) =\sum_j Z(x_j) =\Gamma_K \backslash \mathbb D_m.
$$
where $\mathbb D_m$ is the set of $(z_1, z_2) \in \H^2 \times (\H^-)^2$ which satisfying
$z_1 = x(wz_2)$
 for some $x \in L_m$.

Since there is some $(\gamma_1, \gamma_2) \in \Gamma_K$ with $\det \gamma_1 =\det \gamma_2 =-1$, one has thus
$$
Z(m, \varphi_f) = (\Gamma_0^D(N) \times \Gamma_0^D(N))\backslash \mathbb D_m^+=T_{D, N}(m) .
$$
Here $\mathbb D_m^+ =(\H \times \H) \cap \mathbb D_m$. This proves the lemma.
\end{proof}

\begin{theorem} For $\varphi_f =\cha(\hat{\OO}_D(N))$, one has
$$
I(\tau, \varphi_f \varphi_\infty^{sp} ) = v^{-1} I(g_\tau, \varphi_f\varphi_\infty^{sp}) = \sum_{m=0}^\infty r_{D, N}'(m) q^m
$$
where  $r_{D, N}'(0) =1$, and for $m >0$
$$
r_{D, N}'(m) = \frac{1}{ \vol(X_0^D(N), \Omega_0)} \deg T_{D, N}(m)
$$
as in the introduction.

\end{theorem}
\begin{proof}  Write
$$
I(\tau, \varphi_f \varphi_\infty^{sp} )  = \sum_{m=0}^\infty c(m) q^m.
$$
By \cite[Section 4.8]{KuIntegral}, one has $c(0) =1$ and for $m >0$
$$
c(m)=(\vol(X_K, \Omega^2))^{-1} \int_{ Z(m, \varphi_f) } \Omega.
$$
Clearly,
$$
\vol(X_K, \Omega^2) = \frac{1}{2}  \frac{1}{4\pi^2}  \int_{X_0^D(N) \times X_0^D(N)} \frac{dx_1\wedge dy_1}{y_1^2} \wedge \frac{dx_2\wedge dy_2}{y_2^2}
 =\frac{1}2 \vol(X_0^D(N), \Omega_0)^2.
$$
One the other hand, $\Omega = \pi_1^*(\Omega_0) + \pi_2^*(\Omega_0)$.  So Lemmas \ref{lem5.2} gives
\begin{align*}
c(m)&= \int_{T_{D, N}(m)}( \pi_1^*(\Omega_0) + \pi_2^*(\Omega_0)
\\
   &= 2\int_{T_{D, N}(m)} \pi_1^*(\Omega_0)
   \\
    &= 2 \deg T_{D, N}(m) \int_{X_0^D(N)} \Omega_0.
\end{align*}
So $c(m) = r_{D, N}'(m)$ as claimed.
\end{proof}

{\bf Proof of Theorems \ref{theo1.3}, \ref{theo1.4} and \ref{theo1.5}}:  Now Theorems \ref{theo1.3}, \ref{theo1.4} and \ref{theo1.5} follows the same way as Theorem \ref{theo1.1}. We  verify Theorem \ref{theo1.4} and leave the others to the reader.  Let $V^{(1)} =V(D)$ and $V^{(2)} =V(Dp)$ as in the notation of proof of Theorem \ref{theo1.1}, and let $\varphi^{(i)} =\prod_l \varphi_l^{(i)} \in S(V^{(i)}(\A)$ be as follows. For $l \nmid p\infty$, we identify $\OO_D(N)_l$ with $\OO_{Dp}(N)_l$ and define $\varphi_l^{(i)} = \cha(\OO_D(N)_l)$. Let
$$
\varphi_\infty^{(1)} =\varphi_\infty^{ra},  \quad \varphi_\infty^{(2)} =\varphi_\infty^{sp}.
$$
Finally, let
$$
\varphi_p^{(1)} = -\frac{2}{p-1} \varphi_{p, 0}^{sp} + \frac{p+1}{p-1}\varphi_{p, 1}^{sp}, \quad \varphi_p^{(2)} = \varphi_{p}^{ra}.
$$
Then  $\varphi^{(1)}$ and $\varphi^{(2)}$ match by the results in Section \ref{sect:matching}. So one has  by Proposition  \ref{prop3.5}
$$
I(\tau, \varphi^{(1)}) =I(\tau, \varphi^{(2)}).
$$
Comparing the $m$-coefficients of the both sides, one proves Theorem \ref{theo1.4}.


\begin{thebibliography}{[BCDT}




\bibitem[Gr]{Gr2} B. Gross, Heights and the special values of L-series. Number theory (Montreal, Que., 1985), 115--187, CMS Conf. Proc., 7, Amer. Math. Soc., Providence, RI, 1987.


\bibitem[GK]{GK} B. Gross and K. Keating,  On the intersection of
modular correspondences, Invent. Math., 112 (1993), 225--245.

\bibitem[Gr]{Gross}   B. Gross,  On canonical and quasi-canonical liftings, Invent.  Math.,
84 (1986), 321-326




\bibitem[Ku1]{KuSplit} {\em  S. Kudla},  Splitting metaplectic covers of dual reductive pairs. Israel J. Math. 87 (1994), 361¨C401.


\bibitem[Ku2]{KuIntegral} S. S. Kudla, Integrals of Borcherds Forms,Compositio Mathematica,  137(2003),  293-349.


\bibitem[KRY]{KRYComp}   S. Kudla, M. Rapoport, and T.H. Yang,
 Derivatives of Eisenstein Series and Faltings heights,   Comp. Math., {\bf 140 }(2004),   887-951.

 \bibitem[KR1]{KR1}    S. Kudla and S. Rallis,  On the Weil-Seigel formula,  Crelle, 387(1988), 1-61.

 \bibitem[KR2]{KR2}    S. Kudla and S. Rallis,  On the Weil-Seigel formula II,  Crelle, 391(1988), 65-84.



\bibitem[Mi]{Miy}  T. Miyake,  modular forms,  Springer, New York,  1989.


\bibitem[Si]{siegel formula}   C.L. Siegel,  Uber die analytische Theorie der quadratischen Formen, Ann. Math. 36 (1935), 527-606.

\bibitem[Vi]{Vig}      M. F. Vigneras, Arithm\'etique des alg\'ebres de quaternions, Lecture Notes in Math., no. 800, Springer Verlag, 1980.

\bibitem[Wed]{Wed} T. Wedhorn,  Genus of the endomorphisms of a
supersingular elliptic curve, Asterisque 312(2007), 25--47.


\bibitem[We]{Weil}  \em A.  Weil,   Sur la formule de Siegel dans la th¨¦orie des groupes classiques. (French) Acta Math. 113(1965),  1--87.

\bibitem[Ya1]{YaDensity} T. H. Yang,  An explicit formula for local densities of quadratic
forms, J. number theory, 72 (1998), 309--356.


\bibitem[Ya2]{YaGeneral}  T.H. Yang,  Arithmetic intersection and Faltings' height, Asian J. Math., to appear.

\end{thebibliography}
\end{document}